\documentclass[preprint,12pt]{article}
\usepackage{geometry,comment}
\usepackage{setspace}
\usepackage{graphicx}
\usepackage{epstopdf}
\usepackage{amssymb}
\usepackage{amsmath}
\usepackage{amsthm}
\newtheorem{pro}{Proposition}
\newtheorem{thm}{Theorem}
\newtheorem{lem}{Lemma}
\newtheorem{rmk}{Remark}
\newtheorem{emp}{Example}

\title{Equilibrium, social welfare, and revenue in an infinite-server queue}

\author{Yan Su$^{a,}$\footnote{Corresponding author}, Junping Li$^{b}$\\
\small
{\it $^{a}$School of Management, Nanjing University of
Posts and Telecommunications, Nanjing, 210003, PR China}\\
\small {\it email: ysumath@163.com} \\
\small
{\it $^{b}$School of Mathematics and Statistics, Central South University, Changsha, 410083, PR China}\\
\small {\it email: jpli@mail.csu.edu.cn}\\
}

\begin{document}
\maketitle

\begin{abstract}
Motivated by the impact of emerging technologies on (toll) parks, this paper  studies a problem of  equilibrium, social welfare, and revenue for an infinite-server queue. More specifically, we assume that a customer's utility consists of a positive reward for receiving service minus a cost caused by the other customers in the system.
In the observable setting, we show the existence, uniqueness, and expressions of the individual threshold, the socially optimal threshold, and the optimal revenue threshold, respectively. Then, we prove that
the optimal revenue threshold is smaller than the socially optimal threshold, which is smaller than the individual one.
Furthermore, we also extend
the cost functions to any finite polynomial function with non-negative coefficients.
In the unobservable setting,
we derive the joining probabilities of individual and optimal revenue.
Finally, using numerical experiments, we complement our results and compare the social welfare and the revenue under these two information levels.

\end{abstract}





\section{Introduction}

Infinite-server queues are an important class of stochastic service systems.
In the real world, many service systems can be approximated as infinite-server queues,
e.g., large resorts, (toll) parks in urban areas.
Although the interior of these systems could be divided into several different queueing models or networks, it is still a reasonable approximation to view them as an infinite-server queue as a whole.
Therefore, the literature on infinite-server queues is very extensive; see, for instance, \cite{Altman1,Blom1,Brown1,Collings1,Fakinos,Mirasol1,Shanbhag1,Pang1} and the extensive references therein.

In the past few decades,
studying queueing systems from an economic perspective has become increasingly prominent. More specifically,
a specific reward-cost structure is imposed on the queueing system to reflect customers' desire to be served and unwillingness to wait.
Arriving customers are permitted to make decisions about whether to join the queue.
All customers would like to maximize their profit, taking into consideration that all the other customers also have the same goal. Thus, this situation
could be regarded as a game among the customers.
In the following, let us briefly review the development of studying the queueing problems from such an economic analysis.
In a seminal paper, Naor \cite{Naor} introduced the reward and the linear delay cost of customers into the $M/M/1$ queueing model. If the queue length can be accurately observed by the customers, Naor \cite{Naor} gave threshold strategies of the individual equilibrium, the socially optimal welfare, and the optimal revenue  and proposed the idea of levying fees to induce the social optimal strategy.
Edelson and Hilderbrand \cite{Edelson} complemented Naor's research from an unobservable case. Their conclusion shows that the social welfare and the revenue are equal when the queue length is not observed by customers. Therefore, they proposed a method of levying observation fees to make the social welfare and the revenue still coincide when customers' queueing strategy has a threshold type. However, in the case of non-homogeneous costs, the aforementioned results do not always hold.
Hassin \cite{Hassincom} compared Naor's observable model with Edelson and Hilderbrand's unobservable model. The conclusion shows that providing real-time information is not always beneficial to profit maximization of the manager and the social welfare under the profit maximizing admission fee also has the similar results.
Since then, because the strategic queueing models have been widely used in various service industries, more and more scholars have paid attention to
the problem of strategic queue and numerous excellent papers have been published, such as
vacation queues in the transportation industry (Guo and Hassin \cite{Guo2}),
retrial queueing systems with applications in networks (Wang and Zhang \cite{Wangret}, Cui, Su and Veeraraghavan \cite{Cui1}),
double-ended queues in the passenger-taxi service system
(Shi and Lian \cite{Shi}),
priority queues with discriminatory pricing
(Hassin and Haviv \cite{Hassinpro}, Wang, Cui, and Wang \cite{Wangpro}),
queues with uncertain/different information
(Cui and Veeraraghavan \cite{Cui2}, Hassin, Haviv, and Oz \cite{Hassin}, Chen and Hasenbein \cite{Hasenbein}, Liu \cite{Liuc}), etc.
The basic knowledge of strategic queues is summarized in Hassin and Haviv \cite{Hassinbook}. Recently, the book Hassin \cite{Hassin2016}  lists most of the relevant literature. Interested readers can refer to it and
the extensive references therein.

However, to
the best of our knowledge, the results on equilibrium customer behavior, social welfare, and revenue in an infinite-server queue have not been derived.
Models with infinite servers to approximatively characterize and analyze real problems arise in various situations in practice.
An example of the infinite-server queue may be illustrated by the decision making of tourists in modern parks.
In modern parks, congestion problems occur from time to time due to the centralized travel of people.
For example, in Fantawild (Disneyland, Universal studios) of China,
it is always reported that there are too many people staying in the park during the holidays,
and in urban (toll) parks located in densely populated metropolitan areas, we also usually see a large number of people traveling on weekends.
It is no difficult to find that whether tourists are willing to enter the park has a lot to do with the number of people staying in the park.
An intuitive feeling is that
the more the people stay in the park, the more  reluctant tourists are to join it.
The reason is that
according to the empirical (expected) information or the real-time information provided by the park on the mobile platform (the bulletin board), rational tourists will judge whether it is worth entering the park
and  their individual utilities are negatively correlated with the number of people in the park.
Based on this
phenomenon, we could model these parks as an infinite-server queue and quantify
tourists' behavior by using the game theory and analyze the equilibrium, the social welfare, and the revenue under different information levels, so as to provide some valuable advice to the public.

In the traditional literature,
we usually see that some basic hypotheses of the queueing model have the following salient characteristics:
the customer's reward is assumed to be $R>0$ and  customers' own cost is positively correlated with their sojourn time.
In the context of infinite-server queues, the customer's reward $R>0$ can have the same interpretation and thus, remain consistent with the previous literature. However,
the total costs of customers are assumed to be positively correlated with information on the number of customers in the system.
An interesting practical explanation is that in the park example,
this assumption is able to reflect the impact of the park population on tourist’s satisfaction in modern parks.
Using this new reward-cost structure, there are several contributions in the present paper. First, according to whether to  announce the real-time number of people, we divide the problem into the observable model and the unobservable model. For these two cases, we  analyze the individual equilibrium, the optimal social welfare, and the optimal revenue of the infinite-server queue and gives computable expressions for these optimal policies. Furthermore, we theoretically show the relationship of these optimal strategies and make some monotonic analyses.
Finally, we numerically compare the social
welfare and the revenue with different thresholds and information levels, and some valuable suggestions for
the system administrator are also presented.

The rest of this article is arranged as follows.
In Section 2, we give a detailed description of the model and the reward-cost
structure. Sections 3 and 4 is devoted to
the observable and unobservable cases of the model.
Section 5 shows numerical analyses including a
mini example which gives a simple operation procedure for calculating each quantity.
The proofs of the main
results are postponed to Section 6.
The paper ends presenting some conclusions and potential research directions in Section 7.

\section{Formulation and Preliminaries}
Following the background  described in the Introduction, here we  consider an infinite-server queue.
We assume that customers are homogeneous and arrive at the system according to a Poisson process with potential
arrival rate $\Lambda$.
The sojourn times of all customers in the system  are independent
and  follow a common general
distribution function  $B(x)$ with  a mean of $1/\mu$.
A customer's utility is assumed to consist of a reward for receiving service minus a cost caused by the other customers  in the system. More specifically, on successful completion of service, the service reward $R>0$ is the same for all customers.
If the system administrator announces the real-time number of customers in the
system, the costs of arriving customers are  $C_1 N+ C_2 N^2$ when there are $N$ customers in the system.
There are many practical explanations when $C_1$ and $C_2$  take different values.
\begin{itemize}
  \item[1.] If $C_1>0$ and $C_2=0$, we have $C_1 N+ C_2 N^2=C_1 N$, which means that
     the cost is a linear function of the current number of customers in the system. Thus, this corresponds to the risk-neutral customers.
  \item[2.] If $C_1=0$ and $C_2>0$, we have $C_1 N+ C_2 N^2=C_2 N^2$, which implies that
      the cost is a quadratic function with respect to the real-time number of customers in the system.
      This represents the risk-averse customers.
\end{itemize}
If the system does not announce the real-time number of customers, we assume that customers use the expected information to estimate the number of customers in the system. Therefore, the costs of customers are $C_1 \mathbb{E}(L)+ C_2 \mathbb{E}(L)^2$, where $\mathbb{E}(L)$ is the average number of customers in the system.
Similarly, we could get corresponding interpretations when customers use the cost structure $C_1 \mathbb{E}[L]+ C_2 \mathbb{E}[L]^2$. Moreover, if $C_1=C (1+\mu A''(1)\int_{0}^\infty [1- B(x)]^2 dy)$ and $C_2=C$, we also have $C_1 \mathbb{E}[L]+ C_2 \mathbb{E}[L]^2=C \mathbb{E}[L^2]$, where $A(z)=\mathbb{E}[z^X]$  (see Holman, Chaudhry and  Kashyap \cite{Holman}).  This expression indicates that customers' costs are linear to the second moment of queue length. In the following sections, we only
require $\max\{C_1,C_2\}>0$.
Therefore, the above statements are only a special case. We also investigate that the cost structures are any finite polynomial function with non-negative coefficients, but for brevity, if
the extended conclusions can be obtained in the same way, we will only state them in remarks.
Besides the individual utility,
the additive social utility composed of
the sum of individual utilities
and the revenue composed of long-term gains from monopolist pricing are also analyzed in the following sections.

\section{The Observable Model}\label{tom}

\subsection{Individual Equilibrium}\label{sec11}
In the observable setting, we assume that
the real-time number of customers in the system are always  posted on the bulletin board and
all rational customers can clearly know this information before deciding whether to join the system.
According to the reward-cost structure,
an arriving customer who finds  $n$ customers in the system joins the queue with the individual utility
$ R-  (C_1 n+ C_2 n^2)$
and balks with the individual utility 0. It follows from $R >0$ that rational customers will
join the system if and only if the individual utility is
nonnegative. Therefore, the maximum integer $n_e-1$ that the customers decide to join the system will satisfy the following  two inequalities
\begin{equation}\label{inb}
   R-  [C_1 (n_e-1)+ C_2 (n_e-1)^2] \geq 0,
\end{equation}
\begin{equation}\nonumber
R-  [C_1 n_e+ C_2 n_e^2] <0.
\end{equation}
Solving the above two inequalities, we have the following theorem.
\begin{thm}
In the observable infinite-server queue,
there exists a unique equilibrium strategy ${n}_e=\max\bigl\{N: N \leq \frac{2C_2- C_1 + \sqrt{( C_1)^2+ 4  C_2 R}}{2  C_2 } \bigl\}$ such that customers join the system if and only if $n < {n}_e$.
\end{thm}

\subsection{Social Optimality}
Now, we could consider the problem of maximizing the expected total net gain of all customers per time unit, i.e, the socially optimal welfare. Since the actual (long-run) joining rate is an important index, we
must proceed in a different mode from the individual equilibrium.

Denote $\rho=\frac{\Lambda}{\mu}$ and let $\mathbb{P}(N=j)$ $(j=1,2,\dots,n)$ be the stationary distribution of the $M/G/n/n$ queue.
Note that when all customers consistently use balking strategies $n$ (join the system if and only if the number of customers is less than $n$), we can regard this process as an $M/G/n/n$ queue.
It follows from the results of $M/G/n/n$ queue (see  Fakinos \cite{Fakinos} or Shortle,  Thompson,  Gross, and  Harris \cite{Gross}) that
\begin{equation}\label{4}
  \mathbb{P}(N=j)= \frac{(\frac{\Lambda}{\mu})^j \frac{1}{j!}}{\sum_{k=0}^{n}(\frac{\Lambda}{\mu})^k \frac{1}{k!}}=\frac{ \frac{\rho^j}{j!}}{\sum_{k=0}^{n}\frac{\rho^k}{k!}},
\end{equation}
\begin{equation}\label{5}
  \mathbb{E}(L_n)=\frac{ \sum_{k=0}^{n} k \frac{\rho^k}{k!}}{\sum_{k=0}^{n}\frac{\rho^k}{k!}}=\frac{\rho \sum_{k=0}^{n-1} \frac{\rho^k}{k!}}{\sum_{k=0}^{n}\frac{\rho^k}{k!}},
\end{equation}
where $\mathbb{E}(L_n)$ is the expected queue length of the $M/G/n/n$ queue.
Then, if all customers consistently use the balking strategy $n$, the actual joining
rate of customers  is
$$\Lambda  \mathbb{P}(N < n)=\Lambda\frac{ \sum_{k=0}^{n-1} \frac{\rho^k}{k!}}{\sum_{k=0}^{n}\frac{\rho^k}{k!}}= \mu \mathbb{E}(L_n).$$

Let $S^r(n)$ be the expected total net gain per time unit under balking strategy $n$. Using the PASTA property (see Wolff \cite{Wolff}), we arrive at the following expressions of $S^r(n)$:
\begin{eqnarray}
   S^r(n) &=&  \Lambda \sum_{m=0}^{n-1} \mathbb{P}(N=m)\biggl[R- (C_1 m+ C_2 m^2)\biggl]\label{bet1}\\
  &=&  \Lambda \biggl[ R \mathbb{P}(N<n)- \bigl(C_1  \sum_{m=0}^{n-1} m \mathbb{P}(N=m)+ C_2 \sum_{m=0}^{n-1} m^2 \mathbb{P}(N=m)\bigl)\biggl] \nonumber \\
  &=& \mu R \mathbb{E}(L_{n})- \mu \rho  \biggl[C_1  \frac{\sum_{m=0}^{n-1} m \frac{\rho^m}{m!}}{\sum_{m=0}^{n} \frac{\rho^m}{m!}}+ C_2 \frac{\sum_{m=0}^{n-1} m^2 \frac{\rho^m}{m!}}{\sum_{m=0}^{n} \frac{\rho^m}{m!}}\biggl]. \label{sr}
\end{eqnarray}

According to the expression of (\ref{sr}), we could get the following key results about the socially optimal welfare. The proof can be found in Section \ref{mainr}.

\begin{thm}\label{thm2}
In the observable infinite-server queue, there exists
a unique socially optimal threshold  strategy
\begin{equation}\label{nse}
n_s=\max\biggl\{N: \frac{ \rho \sum_{i=1}^2 C_i \biggl[ \frac{\sum_{m=0}^{N-1} m^i \frac{\rho^m}{m!}}{\sum_{m=0}^{N} \frac{\rho^m}{m!}}-\frac{\sum_{m=0}^{N-2} m^i \frac{\rho^m}{m!}}{\sum_{m=0}^{N-1} \frac{\rho^m}{m!}}\biggl]}{\mathbb{E}(L_{N} )-\mathbb{E}(L_{N-1})} \leq  R  \biggl\}
\end{equation}
such that  $S^r(n)$ is strictly increasing in $n$ when $n \leq n_s$ and strictly decreasing in $n$ when $n>n_s$.
Furthermore, $n_s$ is decreasing in $\rho$.
\end{thm}

\begin{rmk}\label{rmkne}
There exists an intuitive explanation for the relationship between $n_s$ and $\rho$.
As $\rho$ increases, if $n_s$ remains the same,
the number of customers in the system will be {\it stochastically increasing} in $\rho$. This, together with the individual utility $ R-  (C_1 n+ C_2 n^2)$, implies  that the (long-run) average marginal net utility for each arriving customer will become very small.
At this point, the optimal threshold $n_s$ should be reduced to
increase the average net utility of each customer in the system and further increase the additive social
utility.
This interpretation is consistent with the monotonicity of $n_s$ with respect to $\rho$.
On the other hand, the smaller $\rho$ is, the greater the probability that arriving customers will see a small number of customers in the system. Simple calculations yield
\begin{equation}\nonumber
  \lim_{\rho\to 0}\frac{\rho  \biggl[ \frac{\sum_{m=0}^{n} m^i \frac{\rho^m}{m!}}{\sum_{m=0}^{n+1} \frac{\rho^m}{m!}}-\frac{\sum_{m=0}^{n-1} m^i \frac{\rho^m}{m!}}{\sum_{m=0}^{n} \frac{\rho^m}{m!}}\biggl]}{\mathbb{E}(L_{n+1} )-\mathbb{E}(L_{n})}=n^i.
\end{equation}
This combining with (\ref{inb}) and  (\ref{nse}) implies that when $\rho$ tends to 0,
$n_s$ and $n_e$ are getting closer and closer and therefore
allowing customers with the positive individual utility value to enter the system will increase the social welfare.
\end{rmk}

\begin{rmk}
$\mathbf{(a)}$
Letting $\Lambda' >\Lambda$, we consider
a new strategy such that if the number of customers in the system is less than $n_s$, the new system with  parameter $\Lambda'$ (All the other parameters are  assumed to be the same) allows customers to enter with probability $\frac{\Lambda}{\Lambda'}$. According to the decomposability of the Poisson flow and the expression (\ref{sr}), we see that under this new  strategy,
the social welfare of  this new system is also equal to $S^r(n_s)$.
Note that we can regard this social welfare problem
of the infinite-server queue as a particular (long-run) average reward model in the theory of the Markov decision processes (MDPs). Thus,
according to the results of the MDPs (see Chapter 11 in Puterman \cite{Puterman} or Feinberg and Yang \cite{Feinberg}),
the deterministic stationary optimal strategy (or optimal pure strategy) always exists, which means that $S^r(n_s)$ is increasing in $\Lambda$.

$\mathbf{(b)}$
For fixed $n < n_e$, let $\mathbb{P}[N(\rho)=j]$, $j=1,2,\dots,n$, be the stationary distribution of the $M/G/n/n$ queue with parameter $\rho$. Using the method of the sample path comparison, we easily have $N(\rho)\leq_{st}  N(\rho')$ when $\rho<\rho'$, where $\leq_{st} $ is the usual stochastic order (see M\"uller and Stoyan \cite{Muller} or Keilson and Kester \cite{Keilson}).  Therefore, for the decreasing sequence $R- (C_1 n+ C_2 n^2)$ with respect to $n$, we have
$$\sum_{m=0}^{n-1} \mathbb{P}(N(\rho)=m)\biggl[R- (C_1 m+ C_2 m^2)\biggl]>\sum_{m=0}^{n-1} \mathbb{P}(N(\rho')=m)\biggl[R- (C_1 m+ C_2 m^2)\biggl],$$
which, together with (\ref{bet1}), implies that $\frac{S^r(n)}{\Lambda}$ is strictly decreasing in $\rho$. This means that $S^r(n_s)$ is strictly increasing in $\mu$.
\end{rmk}

\subsection{The System's Revenue}\label{Stackelberg}
In this subsection, we introduce a price $P_o$ to study
the system's revenue maximizing problem. Because customers respond to $P_o$,
we model the interaction between the
system administrator and
the customers as a Stackelberg game, where the
system administrator is the leader and customers are the followers. The goal of the system administrator is to maximize its
revenue while anticipating customers’ equilibrium
strategies.
Similar to the traditional analysis (see Section 2.4 of Hassin and Haviv \cite{Hassinbook}),
under a balking strategy $n$, the best price is $P_o=R-  [C_1  (n-1)+ C_2  (n-1)^2]$ and thus, the expected total net revenue, $S^r_m(n)$, can be expressed as follows:
\begin{eqnarray}\label{rts}
   S^r_m(n) &=& \Lambda  \mathbb{P}(N < n)P_o\nonumber \\
  &=& \mu  \mathbb{E}(L_{n}) [ R-  (C_1  (n-1)+ C_2  (n-1)^2)]. \label{rev1}
\end{eqnarray}
According to this expression,
we are able to obtain the following theorem about the optimal revenue. The proof can be found in Section \ref{mainr}.

\begin{thm}\label{thm3}
In the observable infinite-server queue,
there exists
a unique optimal threshold strategy
\begin{equation}\label{nrev}
  n_m=\max\biggl\{N: \frac{ \sum_{i=1}^2 C_i ( \mathbb{E}(L_{N})(N-1)^{i}- \mathbb{E}(L_{N-1})(N-2)^{i})]}{ \mathbb{E}(L_{N})- \mathbb{E}(L_{N-1})} \leq  R  \biggl\}
\end{equation}
such that
$S^r_m(n)$ is strictly increasing in $n$ when $n \leq n_m$ and strictly decreasing in $n$ when $n>n_m$. Moreover,
$n_m$ is increasing in $\rho$ and the optimal price for the system administrator  is $\widetilde{P}_o=R-  [C_1 (n_m-1)+ C_2 ({n_m-1} )^2].$
\end{thm}

\begin{rmk}
The relationship between $n_m$ and $\rho$ has an interesting interpretation.
As $\rho$ increases, increasing $n_m$ will allow more customers to be served, thereby increasing the (long-run) revenue of the system. This reflects the small profit but quick turnover strategy often used in economics.
\end{rmk}

Having obtained the optimal threshold strategies of individual equilibrium, social welfare, and revenue,
now, we could compare
the relationship of size between them. The following results show that for the observable case, if the customers' costs are positively correlated with the real-time number of customers in the system, the three thresholds are generally different and the unequal relationship is consistent with the traditional conclusion found by Naor \cite{Naor}.

\begin{thm}\label{comnnn}
In the observable infinite-server queue, we have
$n_m \leq  n_s \leq n_e$.
\end{thm}
\begin{proof}
It follows from (\ref{inb}) and (\ref{bet1}) that
$n_s \leq n_e$ is straightforward, therefore we just need to show $n_m \leq  n_s$.
Since
$$\lim_{\rho\to \infty}\frac{\rho  \biggl[ \frac{\sum_{m=0}^{n} m^i \frac{\rho^m}{m!}}{\sum_{m=0}^{n+1} \frac{\rho^m}{m!}}-\frac{\sum_{m=0}^{n-1} m^i \frac{\rho^m}{m!}}{\sum_{m=0}^{n} \frac{\rho^m}{m!}}\biggl]}{\mathbb{E}(L_{n+1}) -\mathbb{E}(L_{n})}=n^i(n+1)-n(n-1)^i$$
and
$$\lim_{\rho\to \infty} \frac{   \mathbb{E}(L_{n+1})n^{i}- \mathbb{E}(L_{n})(n-1)^{i}}{ \mathbb{E}(L_{n+1})- \mathbb{E}(L_{n})}=n^i(n+1)-n(n-1)^i,$$
by the definition of (\ref{nse}) and (\ref{nrev}), we have  that
$n_m=n_s$ when $\rho \to \infty$.
Using the results of Theorem \ref{thm2} and  Theorem  \ref{thm3}, we know  that $n_s$ is decreasing in $\rho$ while $n_m$ is increasing in $\rho$, which immediately  indicates  $n_m \leq n_s$ for $\rho<\infty$.
\end{proof}

\begin{rmk}\label{rmknm}

The first inequality of Theorem \ref{comnnn} shows that if the system administrator wants to
maximize its own revenue, the entrance price
will be too high such that the  social utility of the system can not be optimal. However, it follows from the proof of Theorem \ref{comnnn} that the capacity gap between the socially optimal strategy and the revenue-maximizing strategy gradually decreases as $\rho$ increases.
An interesting explanation is that as $\rho$ increases,
there will be more customers left in the system and the profits of all new arrivals are closer to the system's pricing, which brings the socially optimal welfare closer to the optimal revenue.
The second inequality indicates us that
under the fully free condition, the system is generally not able to achieve the social optimum.
Thus, appropriate
tolls  are still a good way to reach the  socially optimal threshold.

\end{rmk}

\section{The Unobservable Model}\label{tuom}

In the unobservable model,
we assume that customers can not obtain the real-time number of customers in the system upon arrival.
Customers make decisions based on the information of the system including $\Lambda$, $\mu$, $R$, and the cost structure $C_1 \mathbb{E}(L)+ C_2 \mathbb{E}(L)^2$, where $\mathbb{E}(L)$ is the average number of people in the system.
To consider a symmetric equilibrium, we suppose that customers will join the system with probability $q$ $(0 \leq q \leq 1$) upon arrival.
In this section, we first derive the equilibrium strategy with no price setted. When the system administrator chooses a desired threshold
$n$ and sets the maximum price $P_u$ to ensure this threshold, like in Section \ref{Stackelberg}, the results of this Stackelberg game is
$P_u=R-  [C_1 \mathbb{E}(L)+ C_2 \mathbb{E}(L )^2]$.
Since the system is unobservable, the social welfare and
the revenue have the same expressions, thus we don't need to distinguish them in the following study.

\subsection{Equilibrium}
When $R- [C_1 \rho+C_2 \rho^2] \leq 0$,
suppose that the equilibrium strategy of the customer to join the system is  $q_e$, $0 \leq q_e \leq 1$, then $q_e$ should satisfies
\begin{eqnarray}\label{un1}
  0& =&R-  (C_1 \mathbb{E}(L)+ C_2 \mathbb{E}(L)^2)\nonumber \\
  &= &R- [C_1 (\rho q_e)+C_2 (\rho q_e)^2],\nonumber
\end{eqnarray}
where $\mathbb{E}[L]$ is the expected queue length of $M/G/\infty$ with the effective arrival rate $\Lambda q_e$ and the service rate $\mu$.
Solving the above equation, we see that $q_e=\frac{- C_1 + \sqrt{ C_1 ^2+ 4 R  C_2 }}{2  C_2 \rho}$ when
$R \leq C_1 \rho+C_2 \rho^2$,
which yields the following theorem immediately.
\begin{thm}\label{th7}
In the unobservable infinite-server queue,
the unique equilibrium strategy
for customers, denoted by $q_e$, is given as follows:
\begin{description}
  \item[(a)] If $R- (C_1 \rho+C_2 \rho^2) \geq 0$, $q_e=1$.
  \item[(b)] If $R- (C_1 \rho+C_2 \rho^2) < 0$, $q_e=\frac{- C_1 + \sqrt{ C_1 ^2+ 4 R  C_2 }}{2  C_2 \rho}.$
\end{description}
\end{thm}

\subsection{Revenue (Social) Optimality}

Like in the observable setting,
when the system administrator sets an entrance fee $P_u$, for individual
decision making, this is equivalent to reducing the service benefit from $R$ to $R-P_u$, which changes the equilibrium probability of joining the system.
Let $q(P_u)$  denote the joining probability associated with a given fee $P_u$ and  without confusion, we use $q$ to represent. Then, we have the expression  of revenue
\begin{equation}\label{qrev}
  S(q) = q\Lambda  [R- (C_1 (\rho q)+C_2 (\rho q)^2)],
\end{equation}
where $P_u=R- [C_1\rho q+C_2 (\rho q)^2]$.
When  $R \leq (2 \rho C_1 +3 \rho^2 C_2 )$,
by differentiating $S(q)$ with respect to $q$ and finding the roots that meet the restriction conditions,
we have
$$q=\frac{- C_1  + \sqrt{C_1^2+ 3 R  C_2}}{3  C_2 \rho }.$$
Summarizing the above discussions, we could naturally
develop the following theorem.
\begin{thm}\label{th8}
In the unobservable infinite-server queue,
let $\widetilde{q}$ be the optimal joining probability of revenue and  $\widetilde{P}_u$ be an optimal entrance price, then the following statements hold.
\begin{description}
  \item[(a)] If $R \geq (2 \rho C_1 +3 \rho^2 C_2 )$, we have
      \begin{itemize}
        \item[1.] $\widetilde{q}=1$;
        \item[2.] $S(\widetilde{q})=\mu (R \rho - C_1  \rho^2 -  C_2 \rho^3)$;
        \item[3.] $\widetilde{P}_u=R- C_1 \rho -C_2 \rho^2.$
      \end{itemize}
  \item[(b)]  If $R < (2 \rho C_1 +3 \rho^2 C_2 )$, we have
       \begin{itemize}
        \item[1.] $\widetilde{q}=\frac{- C_1  + \sqrt{C_1^2+ 3 R  C_2}}{3  C_2 \rho }$;
        \item[2.]  $S(\widetilde{q})= \frac{-2C_1^3-9RC_1C_2+(2 C_1^2+6RC_2)\sqrt{C_1^2+3RC_2}}{27 C_2^2}\mu$;
        \item[3.] $\widetilde{P}_u=\frac{6 R C_2+C_1^2-C_1 \sqrt{C_1^2+3RC_2}}{9 C_2}.$
        \end{itemize}
\end{description}

\end{thm}

\begin{rmk}
$\mathbf{(a)}$
It follows from the results of Theorem \ref{th7} and Theorem \ref{th8} that $\widetilde{q} \leq q_e$ and $S(q_e)=0$, which means that for the unobservable case, the system is still  not able to achieve the social optimum under the fully free condition.
In fact, a customer who decides to join the system would impose negative externalities on future arrivals.
Thus, appropriate tolls are conducive to achieving the social optimality.

$\mathbf{(b)}$
It follows from Theorem \ref{th8}(b) that
if $R < (2 \rho C_1 +3 \rho^2 C_2 )$, both $\widetilde{P}_u$ and $S(\widetilde{q})$ are a constant with respect to $\Lambda$ while $S(\widetilde{q})$ is increasing in $R$ and $\mu$. This implies that  the change in the arrival flow of customers does not affect the revenue and there is no need to adjust the entrance price.
If $R \geq (2 \rho C_1 +3 \rho^2 C_2 )$, $\widetilde{P}_u$ is decreasing in $\rho$ but $S(\widetilde{q})$ is increasing in $\rho$ and $R$, respectively.
This shows that  when $R \geq (2 \rho C_1 +3 \rho^2 C_2 )$ and $\rho$ increases, the system administrator should reduce prices to increase the effective arrival rate and thereby increase the overall utility. On the whole, increasing individual reward ($R$) will always increase the optimal social (revenue) utility, which has the analogous conclusions with the observable case.
\end{rmk}

\begin{rmk}\label{rmk7l}

$\mathbf{(a)}$
The relationship of size between $S^r_m(n_m)$ and $S(\widetilde{q})$ can not be completely determined. In fact, if
$ \mathbb{E}(L_{n_m})\leq (n_m-1)$, it follows from (\ref{rev1}) and (\ref{qrev}) that  $S^r_m(n_m) \leq S(\widetilde{q})$. However, as $\rho \to \infty$,
for $R=20,\mu=1,C_1=1,C_2=0$, we have $S(\widetilde{q})=\frac{\mu R^2}{4 C_1}=100$,
$S^r_m(n_m)=\mu n_m [ R-  (C_1  (n_m-1)] > 10 [ 20 - (10-1)]= 110>S(\widetilde{q})$.
Thus, under the profit-maximizing admission price,
whether the system administrator publishes the real-time number of customers needs to be judged according to specific parameters. Similarly, the relationship of size between $S^r_m(n_s)$ and $S(\widetilde{q})$ can not be completely determined.

$\mathbf{(b)}$ For any $n \leq n_e$, we have
\begin{eqnarray*}
\Lambda \sum_{m=0}^{n-1} \mathbb{P}(N=m)\biggl[R- (C_1 m+ C_2 m^2)\biggl]&\geq& \Lambda \sum_{m=0}^{n-1} \mathbb{P}(N=m)\biggl[R- (C_1 (n-1)+ C_2 (n-1)^2)\biggl] \\
  &=& \Lambda \mathbb{P}(N<n)[R- (C_1 (n-1)+ C_2 (n-1)^2)] \\
  &=& \mu \mathbb{E}(L_{n})[R- (C_1 (n-1)+ C_2 (n-1)^2)],
\end{eqnarray*}
which, combining with (\ref{bet1}) and (\ref{rts}), implies that $S^r(n_m)> S^r_m(n_m)$. This shows that
in order to achieve the revenue-maximizing objective,
if the system administrator chooses to release the real-time number of customers, i.e., $S^r_m(n_m) \geq S(\widetilde{q})$,  this behavior for the public should also be actively advocated because $S^r(n_m)\geq S^r_m(n_m)\geq S(\widetilde{q})$ implies  $S^r(n_m)\geq S(\widetilde{q})$.
\end{rmk}

\section{Numerical Comparisons}

In this section, we present some numerical results for both  observable and unobservable models.
We mainly focus on comparing the social welfare and revenue under these two models to gain insight into some valuable results that have been or have not been proven.  Finally, we also give a simple example to calculate each quantity.

In Figure \ref{fig1}, we compare the social welfare with different $\Lambda$ ($\rho$ with $\mu=1$). From the figure, we could observe the following facts.
\begin{itemize}
  \item[1.] $S^r(n_s)$ and $S^r(n_m)$ are increasing in $\Lambda$, respectively. $S^r(n_s)$ is always bigger than $S(\widetilde{q})$.
      In fact, when $C_2=0$, we have $\rho q=\sum_{m=0}^{\infty} m \mathbb{P}(N=m)$ in (\ref{qrev}). This, combining with expression of (\ref{bet1}), shows that
      we can think of the
 unobservable social welfare problems as an observable (long-run) average reward model in the theory of Markov decision processes (MDPs) with the stochastic Markov strategy, see Chapter 11 in Puterman \cite{Puterman}. Note that $n_s$ is the deterministic stationary optimal strategy in this  average reward model, thus we have $S(\widetilde{q}) \leq S^r(n_s)$.
  \item[2.] $S^r(n_e)$  increases first and then decreases with respect to $\Lambda$.  When $\Lambda$ is relatively large, from the figure, a reasonable toll is a better choice to achieve the social optimum, which coincides with the actual strategy adopted. When $\Lambda$ is relatively small, even if there is no charge, $S^r(n_e)$ is closer to the social optimal welfare. The reasons for this phenomenon have been analyzed in Remark \ref{rmkne}.
  \item[3.] When $\Lambda$ gradually increases, $S^r(n_s)$ and $S^r(n_m)$ get closer and closer until they coincide. In fact, it follows from the proof of Theorem \ref{comnnn} that this is caused by   the gradual approach of the two thresholds. Therefore, when  $\Lambda$ is large, the optimal strategy of the system administrator is gradually in line with the goal of  social maximization.

   \item[4.]
   When $\Lambda$ is relatively small, we can see that $S^r(n_m)$ is less than $S(\widetilde{q})$.
     At this time, under the revenue-maximizing
admission fee, not providing the number of customers  is good for social welfare.
When $\Lambda$ is relatively large, $S^r(n_m)>S(\widetilde{q})$. At this time,  providing the real-time number of customers is beneficial to the social welfare.
In short, whether to publish the real-time
number of customers needs depend on the choice of real parameters.
\end{itemize}

In Figure \ref{fig2}, we provide the revenue with different $\Lambda$. Observing the figure, we have the following statements.
\begin{itemize}
\item[1.] $S^r_m(n_m)$ and $S(\widetilde{q})$ are increasing in $\Lambda$, respectively and after a simple judgment, $S(\widetilde{q})$ is a constant when $\Lambda \geq 7.5$.
      From the figure, we can also see that
      $S^r_m(n_s)$ is increasing in $\Lambda$ although we can not prove it theoretically.
      An intuitive reason is that $n_s$ and $n_m$ gradually approach as $\Lambda$ increases.

\item[2.] When $\Lambda$ is relatively small,
  $S^r_m(n_m)<S(\widetilde{q})$, which means that under the revenue-maximizing admission fee, the system administrator has no incentive to publish the real-time number of customers. Meanwhile, we also see that $S^r_m(n_s)<S(\widetilde{q})$, so under the social welfare maximization threshold,
     the unobservable case will also have greater revenue than  the observable case. That is to say,  it is beneficial for the system administrator not to publish  the real-time information in this circumstance.
      \item[3.]
      When $\Lambda$ gradually increases, $S^r_m(n_m)$ and $S^r_m(n_s)$ will get closer and closer and will exceed $S(\widetilde{q})$. Thus, for sufficiently large $\Lambda$,
      the system administrator is willing to publish real-time information under various optimal thresholds. In fact, we can see that there exists $\Lambda$ such that $S^r_m(n_m)>S(\widetilde{q})>S^r_m(n_s)$. Under this condition, the decision of the system administrator is the opposite to the decision with the social welfare-maximizing objective. Therefore, in order to maximize the optimal social welfare,
      it is necessary to induce
      the system administrator to publish
      the real-time information.
\end{itemize}

\begin{figure}
\centering
\includegraphics[width=0.6\textwidth]{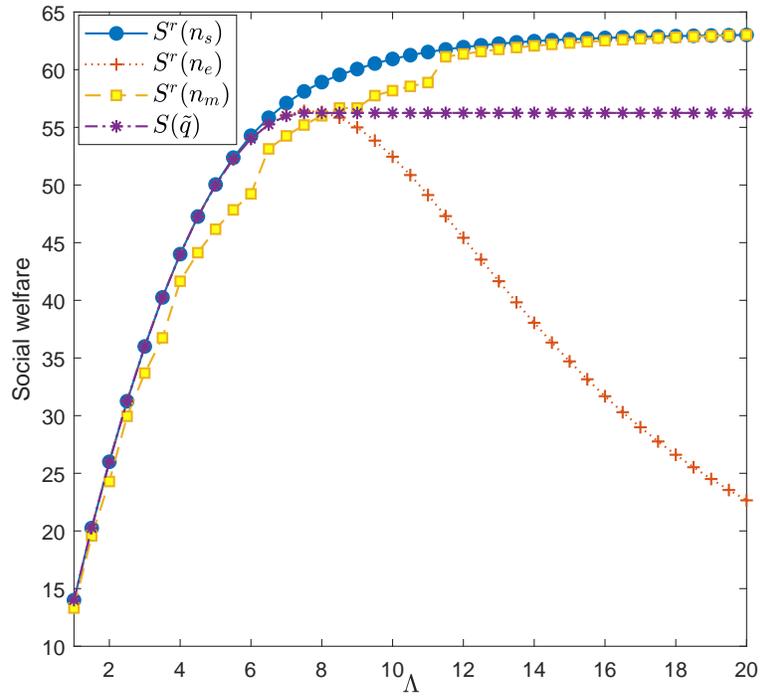}
\caption{Comparison of the social welfare  per time unit vs. $\rho$ for $\mu = 1, R = 15, C_1=1,C_2=0.$}\label{fig1}
\end{figure}

\begin{figure}
\centering
\includegraphics[width=0.6\textwidth]{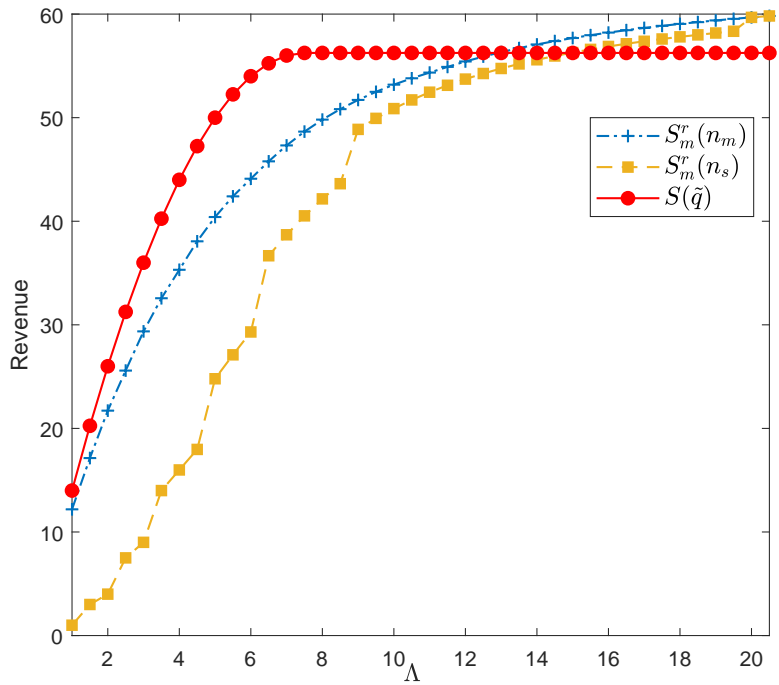}
\caption{Comparison of optimal revenue per time unit vs. $\rho$ for $\mu = 1, R = 15, C_1=1,C_2=0.$}\label{fig2}
\end{figure}

Finally,  we also  complement a concrete example to give a numerical solution of various quantities.

\begin{emp}
Suppose that potential customers arrive at a system 20 persons per minute, the average sojourn time is 60 minutes, the service reward $R=400$ and the cost is $C_1=0,C_2=0.01$.  Then, $\rho=1200$ and using the results of
Section \ref{tom}
and Section \ref{tuom}, we have $n_e=201$, $n_s=116$, $n_m=116$, $q_e=\frac{1}{6}$, $q_s=\frac{\sqrt{3}}{18}$, $S^r(n_s)=517.64$, $S^r(n_e)=2.66$, $\widetilde{P}_o=265.44$, $S^r_m(n_m)=512.71$, $\widetilde{P}_u=266.67$, $S(q_e)=0,S(\widetilde{q})=513.20$.
\end{emp}

In this example, the social optimal threshold and the revenue optimal threshold are the same. Comparing $S^r(n_s)$ and $S^r(n_e)$, whether it is a free system that controls the number of people or a toll system that controls the threshold through charging, the social  welfare will be significantly improved.
\section{Proofs of the Main Results}\label{mainr}

This part is devoted to the proofs of Theorem \ref{thm2} and Theorem \ref{thm3}. To begin with,  we need to state and prove some lemmas and propositions.

\begin{lem}\label{lem4}
Let function $f(x)=b_n x^n+b_{n-1}x^{n-1}+ \dots+ b_1 x + b_0$ $(b_i>0)$ and $g(x)=a_n x^n+a_{n-1}x^{n-1}+ \dots+ a_1 x + a_0$ $(a_i>0)$. Then, if $\frac{b_n}{a_n} \geq\frac{b_{n-1}}{a_{n-1}} \geq \dots \geq \frac{b_{0}}{a_{0}}$,
$\frac{f(x)}{g(x)}$ is increasing in $x$. In particular, when one of the above inequalities is strict, $\frac{f(x)}{g(x)}$ is strictly increasing in $x$.
\end{lem}

\begin{proof}

Through differentiation with respect to $x$, we have
\begin{align}
  \biggl(\frac{f(x)}{g(x)}\biggl)' &= \frac{(\sum_{k=1}^n kb_kx^{k-1})(\sum_{k=0}^n  a_{k} x^k)-(\sum_{k=0}^n b_kx^{k})(\sum_{k=1}^n k a_{k} x^{k-1})}{g(x)^2} \nonumber \\
   &= \frac{\sum_{m=0}^{2n-1}\biggl[\sum_{\substack{i+j=m\\ 0\leq i\leq n-1 \\ 0 \leq j\leq n}}(i+1) b_{i+1} a_j x^m-\sum_{\substack{i+j=m\\ 0\leq i\leq n \\ 0 \leq j\leq n-1}}(j+1) b_i a_{j+1} x^m\biggl]}{g(x)^2} \nonumber \\
      &=  \frac{\sum_{m=0}^{2n-1}\biggl[\sum_{i= \max\{m-n,0\}}^{\min\{m,n-1\}}(i+1) b_{i+1} a_{m-i} x^m-\sum_{i= \max\{m-n,0\}}^{\min\{m,n-1\}}(i+1) b_{m-i} a_{i+1} x^m\biggl]}{g(x)^2}. \label{inee1}
\end{align}
Since  $\frac{b_{i+1}}{a_{i+1}} \geq \frac{b_{m-i}}{a_{m-i}}$ for $i+1>m-i$, we have $[(i+1) b_{i+1} a_{m-i}+(m-i)b_{m-i} a_{i+1}]-[(i+1) b_{m-i} a_{i+1}+(m-i) b_{i+1} a_{m-i}]=(2i+1-m)b_{i+1} a_{m-i}-(2i+1-m)b_{m-i} a_{i+1} \geq 0$.
This immediately implies that
\begin{equation}\label{i2}
\sum_{i= \max\{m-n,0\}}^{\min\{m,n-1\}}(i+1) b_{i+1} a_{m-i} -(i+1) b_{m-i} a_{i+1} \geq 0.
\end{equation}
Thus, $\frac{f(x)}{g(x)}$ is increasing in $x$.
Obviously, when one of the inequalities is strict,  the monotonicity is also strict from the (\ref{inee1}) and (\ref{i2}).
\end{proof}

\begin{pro}\label{mlem3}
Let $f_n(\rho)=\sum_{m=0}^{n} \frac{\rho^m}{m!}$ for $n \geq 0$ and $f_{-1}(\rho)=0$. Then, the following statements hold.
\begin{description}
\item[(a)] $\frac{  f_{n-t}(\rho)f_{n}(\rho)- f_{n-1-t}(\rho)f_{n+1}(\rho)}{ f_{n}(\rho)^2-f_{n-1}(\rho)f_{n+1}(\rho)}$ is strictly increasing in $n$ for $t=1,2, \dots, n$.
\item[(b)] $\frac{ \rho^t f_{n-t}(\rho) f_{n}(\rho) -\rho^t f_{n-1-t}(\rho)f_{n+1}(\rho)}{
f_{n}(\rho)^2- f_{n-1}(\rho)f_{n+1}(\rho)}$  is strictly increasing in $\rho$ for $t=1,2, \dots, n$.
\item[(c)] $\frac{\rho  \biggl[ \frac{\sum_{m=0}^{n} m^i \frac{\rho^m}{m!}}{\sum_{m=0}^{n+1} \frac{\rho^m}{m!}}-\frac{\sum_{m=0}^{n-1} m^i \frac{\rho^m}{m!}}{\sum_{m=0}^{n} \frac{\rho^m}{m!}}\biggl]}{\mathbb{E}[L_{n+1}  ]-\mathbb{E}[L_{n} ]}$
is strictly increasing in $n$ and $\rho$, respectively, for $i=1,2,3,\dots$.
\end{description}
\end{pro}

\begin{proof}
(a) We just need to show $\frac{  f_{n-t}(\rho)f_{n}(\rho)- f_{n-1-t}(\rho)f_{n+1}(\rho)}{ f_{n}(\rho)^2-f_{n-1}(\rho)f_{n+1}(\rho)}>\frac{  f_{n-1-t}(\rho)f_{n-1}(\rho)- f_{n-2-t}(\rho)f_{n}(\rho)}{ f_{n-1}(\rho)^2-f_{n-2}(\rho)f_{n}(\rho)}$, which holds if and only if for $n\geq t+1$,
\begin{eqnarray}\label{ieq}
\begin{aligned}
   &f_{n-t}(\rho)f_{n-1}(\rho)^2+f_{n-2-t}(\rho)f_{n}(\rho)^2+f_{n-1-t}(\rho)f_{n-2}(\rho)f_{n+1}(\rho) \\
    > \ & f_{n-1-t}(\rho)f_{n-1}(\rho)f_{n}(\rho)+f_{n-t}(\rho)f_{n-2}(\rho)f_{n}(\rho)+f_{n-2-t}(\rho)f_{n-1}(\rho)f_{n+1}(\rho).
\end{aligned}
\end{eqnarray}
Denoted the coefficients of the $m$-th power  on both sides of inequality (\ref{ieq}) by
\begin{equation}\label{cG}
G(m)  \triangleq \sum_{\substack{i+j+k=m\\ 0\leq i\leq n-t \\ 0 \leq j\leq n-1 \\ 0 \leq k\leq n-1  }}\frac{1}{i! j!k!}+\sum_{\substack{i+j+k=m\\ 0\leq i\leq n-2-t \\ 0 \leq j\leq n \\ 0 \leq k\leq n  }}\frac{1}{i! j!k!}+\sum_{\substack{i+j+k=m\\ 0\leq i\leq n-1-t \\ 0 \leq j\leq n-2 \\ 0 \leq k\leq n+1  }}\frac{1}{i! j!k!},
\end{equation}
\begin{equation}\label{cH}
H(m)    \triangleq  \sum_{\substack{i+j+k=m\\ 0\leq i\leq n-1-t \\ 0 \leq j\leq n-1 \\ 0 \leq k\leq n  }}\frac{1}{i! j!k!} +\sum_{\substack{i+j+k=m\\ 0\leq i\leq n-t \\ 0 \leq j\leq n-2 \\ 0 \leq k\leq n  }}\frac{1}{i! j!k!}+\sum_{\substack{i+j+k=m\\ 0\leq i\leq n-2-t \\ 0 \leq j\leq n-1 \\ 0 \leq k\leq n+1  }}\frac{1}{i! j!k!}.
\end{equation}
Then, it is clear to see that (\ref{ieq}) holds if $G(m)\geq H(m)$ for $0 \leq m \leq 3n-2-t$ and at least one inequality sign is strictly established.
Next, we show that this condition is always true.

We give a concrete proof for $ 2 \leq t \leq n-2$ and for  $t=1$ or $t=n-1$,  we can also prove it by a similar method. If $m < n-t$, by (\ref{cG}) and (\ref{cH}), we have

\begin{eqnarray*}
   G(m)-H(m) &=& \sum_{\substack{i+j+k=m\\ 0\leq i\leq n-t \\ 0 \leq j\leq n-t \\ 0 \leq k\leq n-t  }}\frac{1}{i! j!k!}+\sum_{\substack{i+j+k=m\\ 0\leq i\leq n-2-t \\ 0 \leq j\leq n-t \\ 0 \leq k\leq n-t  }}\frac{1}{i! j!k!}+\sum_{\substack{i+j+k=m\\ 0\leq i\leq n-1-t \\ 0 \leq j\leq n-t \\ 0 \leq k\leq n-t  }}\frac{1}{i! j!k!} \\
     &\quad &- \biggl[  \sum_{\substack{i+j+k=m\\ 0\leq i\leq n-1-t \\ 0 \leq j\leq n-t \\ 0 \leq k\leq n  }}\frac{1}{i! j!k!} +\sum_{\substack{i+j+k=m\\ 0\leq i\leq n-t \\ 0 \leq j\leq n-t \\ 0 \leq k\leq n-t  }}\frac{1}{i! j!k!}+\sum_{\substack{i+j+k=m\\ 0\leq i\leq n-2-t \\ 0 \leq j\leq n-t \\ 0 \leq k\leq n-t  }}\frac{1}{i! j!k!} \biggl] = 0.
\end{eqnarray*}
If $m \geq n-t$,  taking
differences, we have
\begin{eqnarray}\label{d1}
\sum_{\substack{i+j+k=m\\ 0\leq i\leq n-t \\ 0 \leq j\leq n-1 \\ 0 \leq k\leq n-1  }}\frac{1}{i! j!k!}-\sum_{\substack{i+j+k=m\\ 0\leq i\leq n-t \\ 0 \leq j\leq n-2 \\ 0 \leq k\leq n  }}\frac{1}{i! j!k!} &=& \sum_{i=0}^{n-t} \sum_{\substack{i+j+k=m\\ 0 \leq j\leq n-1 \\ 0 \leq k\leq n-1  }}\frac{1}{i! j!k!}-\sum_{i=0}^{n-t}\sum_{\substack{i+j+k=m\\ 0 \leq j\leq n-2 \\ 0 \leq k\leq n  }}\frac{1}{i! j!k!}  \nonumber  \\
   &=& \sum_{i=0}^{n-t} \sum_{\substack{i+j+k=m\\  j= n-1 \\ 0 \leq k\leq n-1  }}\frac{1}{i! j!k!}-\sum_{i=0}^{n-t}\sum_{\substack{i+j+k=m\\ 0 \leq j\leq n-2 \\  k= n  }}\frac{1}{i! j!k!}.
\end{eqnarray}
Similarly, we arrive at the following two expressions.
\begin{equation}\label{d2}
  \sum_{\substack{i+j+k=m\\ 0\leq i\leq n-1-t \\ 0 \leq j\leq n-2 \\ 0 \leq k\leq n+1  }}\frac{1}{i! j!k!}-  \sum_{\substack{i+j+k=m\\ 0\leq i\leq n-1-t \\ 0 \leq j\leq n-1 \\ 0 \leq k\leq n  }}\frac{1}{i! j!k!}=   \sum_{i=0}^{n-1-t} \sum_{\substack{i+j+k=m\\ 0 \leq j\leq n-2 \\  k= n+1  }}\frac{1}{i! j!k!}-\sum_{i=0}^{n-1-t}\sum_{\substack{i+j+k=m\\ j= n-1 \\ 0 \leq k\leq n  }}\frac{1}{i! j!k!},
\end{equation}

\begin{equation}\label{d3}
  \sum_{\substack{i+j+k=m\\ 0\leq i\leq n-2-t \\ 0 \leq j\leq n \\ 0 \leq k\leq n  }}\frac{1}{i! j!k!}-  \sum_{\substack{i+j+k=m\\ 0\leq i\leq n-2-t \\ 0 \leq j\leq n-1 \\ 0 \leq k\leq n+1  }}\frac{1}{i! j!k!}=   \sum_{i=0}^{n-2-t} \sum_{\substack{i+j+k=m \\  j= n \\ 0 \leq k\leq n  }}\frac{1}{i! j!k!}-\sum_{i=0}^{n-2-t}\sum_{\substack{i+j+k=m\\ 0 \leq j\leq n-1 \\k=n+1  }}\frac{1}{i! j!k!}.
\end{equation}
This, combining with (\ref{cG}) and (\ref{cH}) implies that

\begin{eqnarray}\label{d11n}
\begin{aligned}
G(m)-H(m) & = \biggl[ \sum_{i=0}^{n-t} \sum_{\substack{i+j+k=m\\  j= n-1 \\ 0 \leq k\leq n-1  }}\frac{1}{i! j!k!}-\sum_{i=0}^{n-t}\sum_{\substack{i+j+k=m\\ 0 \leq j\leq n-2 \\  k= n  }}\frac{1}{i! j!k!} \biggl]\\
   & \quad +\biggl[\sum_{i=0}^{n-1-t} \sum_{\substack{i+j+k=m\\ 0 \leq j\leq n-2 \\  k= n+1  }}\frac{1}{i! j!k!}-\sum_{i=0}^{n-1-t}\sum_{\substack{i+j+k=m\\ j= n-1 \\ 0 \leq k\leq n  }}\frac{1}{i! j!k!}\biggl] \\
                    & \quad +\biggl[\sum_{i=0}^{n-2-t} \sum_{\substack{i+j+k=m \\  j= n \\ 0 \leq k\leq n  }}\frac{1}{i! j!k!}-\sum_{i=0}^{n-2-t}\sum_{\substack{i+j+k=m\\ 0 \leq j\leq n-1 \\k=n+1  }}\frac{1}{i! j!k!}\biggl].
\end{aligned}
\end{eqnarray}
By taking differences, we also have that
\begin{eqnarray}
  &\quad& \sum_{i=0}^{n-t} \sum_{\substack{i+j+k=m\\  j= n-1 \\ 0 \leq k\leq n-1  }}\frac{1}{i! j!k!}-\sum_{i=0}^{n-1-t}\sum_{\substack{i+j+k=m\\ j= n-1 \\ 0 \leq k\leq n  }}\frac{1}{i! j!k!} \nonumber\\
 &=&   \frac{1}{(n-t)! (n-1)!(m-2n+t+1)!}-
 \frac{1}{(m-2n+1)! (n-1)!n!} \mathbf{1}_{\{m \geq 2n-1\}}, \label{nt}
\end{eqnarray}
\begin{eqnarray}
  &\quad&  \sum_{i=0}^{n-1-t} \sum_{\substack{i+j+k=m\\ 0 \leq j\leq n-2 \\  k= n+1  }}\frac{1}{i! j!k!}-\sum_{i=0}^{n-2-t}\sum_{\substack{i+j+k=m\\ 0 \leq j\leq n-1 \\k=n+1  }}\frac{1}{i! j!k!} \nonumber\\
 &=&   \frac{1}{(n-1-t)! (m-2n+t)!(n+1)!}-
 \frac{1}{(m-2n)! (n-1)!(n+1)!} \mathbf{1}_{\{m \geq 2n\}}, \label{n1t}
\end{eqnarray}
and
\begin{eqnarray}
  &\quad&  \sum_{i=0}^{n-2-t} \sum_{\substack{i+j+k=m \\  j= n \\ 0 \leq k\leq n  }}\frac{1}{i! j!k!}
  -\sum_{i=0}^{n-t}\sum_{\substack{i+j+k=m\\ 0 \leq j\leq n-2 \\  k= n  }}\frac{1}{i! j!k!} \nonumber\\
 &=&  - \frac{1}{(n-t)! (m-2n+t)!n!} - \frac{1}{(n-1-t)! (m-2n+t+1)!n!} \nonumber\\
 &\quad& \ +\frac{1}{(m-2n)! n!n!} \mathbf{1}_{\{m \geq 2n\}} + \frac{1}{(m-2n+1)! (n-1)!n!} \mathbf{1}_{\{m \geq 2n-1\}}, \label{fd}
\end{eqnarray}
where $\mathbf{1}_B$ is the indicator function of $B$ taking value $1$ if the event $B$ is true and $0$ otherwise and (\ref{fd}) also holds formally for $m=3n-2-t$.
We substitute (\ref{nt}), (\ref{n1t}) and (\ref{fd}) into (\ref{d11n}), and simple calculations show
that
\begin{align*}
  G(m)-H(m) &=  \frac{1}{(n-t)! (n-1)!(m-2n+t+1)!}+\frac{1}{(n-1-t)! (m-2n+t)!(n+1)!}\\
  &\quad -\frac{1}{(m-2n)! (n-1)!(n+1)!} \mathbf{1}_{\{m \geq 2n\}}- \frac{1}{(n-t)! (m-2n+t)!n!}\\
&\quad  - \frac{1}{(n-1-t)! (m-2n+t+1)!(n)!} +\frac{1}{(m-2n)! n!n!} \mathbf{1}_{\{m \geq 2n\}}\\
&=  \frac{3n-t-1-m}{(n-t)! n!(m-2n+t+1)!}+\frac{m-3n+t}{(n-1-t)! (m-2n+t+1)!(n+1)!}\\
&\quad \qquad+\frac{1}{(m-2n)! n!(n+1)!}  \mathbf{1}_{\{m \geq 2n\}} \\
&= \frac{(3n-t-m)(t+1)-(n+1)}{(n-t)! (n+1)!(m-2n+t+1)!}+\frac{1}{(m-2n)! n!(n+1)!} \mathbf{1}_{\{m \geq 2n\}}
\end{align*}
When $n-t\! \leq \! m \! <2n$, we have $(3n-t-m)(t+1)-(n+1)\!> \! (n-t)(t+1)-(n+1)\!=(n-t-1)t-1\!>\!0$.
When $m\geq2n$, it is clear to see that
\begin{eqnarray}
\frac{(3n-t-m)(t+1)-(n+1)}{(n-t)! (n+1)!(m-2n+t+1)!}+\frac{1}{(m-2n)! n!(n+1)!} \mathbf{1}_{\{m \geq 2n\}}>0
\end{eqnarray}
if and only if
\begin{equation}\label{ifb}
\frac{(m-2n+t+1)!(n-t)!}{(m-2n)!n!}-(m-2n)(t+1)>(n+1)-(n-t)(t+1).
\end{equation}
Since
$$\frac{(z+t+1)(z+t)\cdots(z+1)}{n(n-1)\cdots(n-t+1)}-z(t+1)$$
is decreasing in $z$ for $0 \leq z \leq n-2-t$ and for
$m=3n-2-t$,
\begin{eqnarray*}
   &\quad& \frac{(3n-t-m)(t+1)-(n+1)}{(n-t)! (n+1)!(m-2n+t+1)!}+\frac{1}{(m-2n)! (n)!(n+1)!}  \mathbf{1}_{\{m \geq 2n\}}\\
    &=&  \frac{t^2+t}{(n-t)! (n+1)!n!}>0,
\end{eqnarray*}
then, (\ref{ifb}) holds for $m\geq2n$. This yields that  $G(m)>H(m)$ when $m\geq n-t$. Therefore, $(\ref{ieq})$ always holds for any $\rho>0$, which completes the proof.

(b)
First, by rearranging terms, we
arrive at the following more compact representation:
\begin{align*}
   &  \quad \ \rho^t f_{n-t}(\rho) f_{n}(\rho) -\rho^t f_{n-1-t}(\rho)f_{n+1}(\rho)\\
   &= \rho^t \sum_{m=0}^{2n-t}\biggl[\sum_{i=\max\{m-n,0\}}^{\min\{m,n-t\}}\frac{\rho^m}{i!(m-i)!}-\sum_{i=\max\{m-n-1,0\}}^{\min\{m,n-1-t\}}\frac{\rho^m}{i!(m-i)!} \biggl]\\
   &= \rho^t \sum_{m=0}^{2n-t}\biggl[\sum_{i=\max\{m-n,0\}}^{\min\{m,n-t\}}\frac{1}{i!(m-i)!}-\sum_{i=\max\{m-n-1,0\}}^{\min\{m,n-1-t\}}\frac{1}{i!(m-i)!} \biggl]\rho^m.
\end{align*}
Let $F(m,t)=\sum_{i=\max\{m-n,0\}}^{\min\{m,n-t\}}\frac{1}{i!(m-i)!}-\sum_{i=\max\{m-n-1,0\}}^{\min\{m,n-1-t\}}\frac{1}{i!(m-i)!}.$
Then, we have that for $t \geq 0$,
\begin{itemize}
  \item[1.] if $m\leq n-1-t$, $ F(m,t)=0$;
  \item[2.] if $n-t \leq m\leq n$, $ F(m,t)=\frac{1}{(n-t)!(m-n+t)!}>0$;
  \item[3.]  if $n+1 \leq m \leq 2n-t$, $ F(m,t)=\frac{1}{(n-t)!(m-n+t)!}-\frac{1}{(m-n-1)!(n+1)!}>0$.
\end{itemize}
Thus, $\sum_{m=0}^{2n-t}  F(m,t)>0$ and after simple
calculations we also  have that
\begin{itemize}
  \item[1.] if $m\leq n-1-t$ $(m+t \leq n-1)$, $ F(m,t)=0$ and $F(m+t,0)=0$;
  \item[2.] if $m=n-t$ $(m+t=n)$, $ F(m,t)=\frac{1}{(n-t)!(m-n+t)!}>0$ and $F(m+t,0)=\frac{1}{n!(m-n+t)!}>0$;
  \item[3.]  if $n-t+1 \leq m \leq n$ $(n+1 \leq m+t \leq n+t)$, $ F(m,t)=\frac{1}{(n-t)!(m-n+t)!}$ and $ F(m+t,0)=\frac{1}{n!(m-n+t)!}-\frac{1}{(m+t-n-1)!(n+1)!};$
\item[4.]  if $n+1 \leq m \leq 2n-t$ $(n+1+t \leq m+t \leq 2n)$, $ F(m,t)=\frac{1}{(n-t)!(m-n+t)!}-\frac{1}{(m-n-1)!(n+1)!}$ and $ F(m+t,0)=\frac{1}{n!(m-n+t)!}-\frac{1}{(m+t-n-1)!(n+1)!}.$
\end{itemize}
This immediately implies that for
\begin{align}
  &\quad \bigl[(n+1)\cdots(n-t+1)-(m+1-n+t)\cdots(m+1-n)\mathbf{1}_{\{m+1 \geq n+1\}}\bigl]\bigl[(n+1) \nonumber\\
  &\quad   -(m-n+t)\mathbf{1}_{\{m \geq n-t+1\}}\bigl]  -\bigl[(n+1)\cdots(n-t+1)-(m-n+t)\cdots(m-n)\mathbf{1}_{\{m \geq n+1\}}\bigl] \nonumber\\
  &\quad  \bigl[(n+1)-(m+1-n+t)\mathbf{1}_{\{m+1 \geq n-t+1\}}\bigl] \nonumber\\
  =& \ (n+1)\cdots(n-t+1)[(m+1-n+t)\mathbf{1}_{\{m+1 \geq n-t+1\}}-(m-n+t)\mathbf{1}_{\{m \geq n-t+1\}}]\nonumber\\
  &\quad + [(m-n+t)\cdots(m-n)\mathbf{1}_{\{m \geq n+1\}}][(n+1)-(m+1-n+t)\mathbf{1}_{\{m+1 \geq n-t+1\}}]\nonumber\\
  &\quad -[(m+1-n+t)\cdots(m+1-n)\mathbf{1}_{\{m+1 \geq n+1\}}][(n+1)-(m-n+t)\mathbf{1}_{\{m \geq n-t+1\}}] \label{ndca},
\end{align}
we have the following case:
\begin{itemize}
  \item[1.] if $m=n-t$, $(\ref{ndca})=(n+1)\cdots(n-t+1)>0$;
  \item[2.] if $n-t+1\leq m \leq n-1$, $(\ref{ndca})=(n+1)\cdots(n-t+1)>0$ (If $t=1$, there is no such item);
  \item[3.] if $m = n$, $(\ref{ndca})=(n+1)\cdots(n-t+1)-[(t+1)!\mathbf{1}_{\{m+1 \geq n+1\}}][(n+1)-(m-n+t)\mathbf{1}_{\{m \geq n-t+1\}}]>0 $.
   \item[4.] if $m \geq n+1$, $(\ref{ndca})=(n+1)\cdots(n-t+1)-((n+1)-t(2n-t-m))(m-n+t)\cdots(m-n+1)>0.$
\end{itemize}
Note that
for $ m \geq n-t$,
\begin{small}
\begin{align*}
  &\ \ \ \ \ \  \frac{F(m,t)}{ F(m+t,0)}<\frac{F(m+1,t)}{ F(m+1+t,0)}\\
  &\Leftrightarrow
  \frac{(n+1)\cdots(n-t+1)-(m-n+t)\cdots(m-n)\mathbf{1}_{\{m \geq n+1\}}}{(n+1)-(m-n+t)\mathbf{1}_{\{m \geq n-t+1\}}}< \\
  &\  \quad  \quad\quad\quad\quad\quad \quad
  \frac{(n+1)\cdots(n-t+1)-(m+1-n+t)\cdots(m+1-n)\mathbf{1}_{\{m+1 \geq n+1\}}}{(n+1)-(m+1-n+t)\mathbf{1}_{\{m+1 \geq n-t+1\}}}\\
  &\Leftrightarrow  [(n+1)\cdots(n-t+1)-(m\!+1\!-n+t)\cdots(m+1\!-n)\mathbf{1}_{\{m+1 \geq n+1\}}][(n+1)\!-(m-n+t)\mathbf{1}_{\{m \geq n-t+1\}}]\\
  &\ \ \  >[(n+1)\cdots(n-t+1)-(m-n+t)\cdots(m-n)\mathbf{1}_{\{m \geq n+1\}}][(n+1)-(m+1-n+t)\mathbf{1}_{\{m+1 \geq n-t+1\}}],
\end{align*}
\end{small}
which is always true and have been proved in (\ref{ndca}).
Thus, $$\frac{\rho^t F(m,t) \rho^m}{F(m+t,0)\rho^{m+t}}$$
is strictly increasing in $m$ for $m\geq n-t$.
Because $F(m,t)$ and $F(m+t,0)$ are the coefficients of $m+t$ power of $ \rho^t f_{n-t}(\rho) f_{n}(\rho) -\rho^t f_{n-1-t}(\rho)f_{n+1}(\rho)$ and $f_{n}(\rho)^2- f_{n-1}(\rho)f_{n+1}(\rho)$ (when $t=0$), respectively,
it  immediately  follows from Lemma \ref{lem4} that $\frac{ \rho^t f_{n-t}(\rho) f_{n}(\rho) -\rho^t f_{n-1-t}(\rho)f_{n+1}(\rho)}{
f_{n}(\rho)^2- f_{n-1}(\rho)f_{n+1}(\rho)}$  is strictly increasing in $\rho$ for $n \geq t \geq 1$.

(c) For any $i=1,2,3, \dots$, we have
\begin{align*}
 \sum_{m=0}^{n}m^i \frac{\rho^m}{m!}=\rho\sum_{m=0}^{n-1}(m+1)^{i-1}  \frac{\rho^m}{m!} &= \rho\biggl[\sum_{m=0}^{n-1}\biggl( \sum_{k=0}^{i-1}\binom{i-1}{k}m^k \frac{\rho^m}{m!}\biggl)\biggl]\\
 &= \rho\sum_{k=1}^{i-1}\biggl[\sum_{m=0}^{n-1} \binom{i-1}{k}m^k \frac{\rho^m}{m!}\biggl]+\rho
 \sum_{m=0}^{n-1} \frac{\rho^m}{m!}.
\end{align*}
Then, repeating the above arguments, we could derive the following expansion:
$$\sum_{m=0}^{n}m^i \frac{\rho^m}{m!}= \sum_{j=1}^{\min \{n,i\}} a_j \rho^j \sum_{m=0}^{n-j}  \frac{\rho^m}{m!},
$$
where $a_j$ is a positive  constant.
This means that
\begin{align}
   & \quad \ \frac{\rho  \biggl[ \frac{\sum_{m=0}^{n} m^i \frac{\rho^m}{m!}}{\sum_{m=0}^{n+1} \frac{\rho^m}{m!}}-\frac{\sum_{m=0}^{n-1} m^i \frac{\rho^m}{m!}}{\sum_{m=0}^{n} \frac{\rho^m}{m!}}\biggl]}{\mathbb{E}[L_{n+1}  ]-\mathbb{E}[L_{n} ]} \nonumber\\
   &= \frac{ \frac{\sum_{j=1}^{\min \{n,i\}} a_j \rho^j \sum_{m=0}^{n-j}  \frac{\rho^m}{m!}}{f_{n+1}(\rho)}-\frac{\sum_{j=1}^{\min \{n-1,i\}} a_j \rho^j \sum_{m=0}^{n-1-j}  \frac{\rho^m}{m!}}{f_{n}(\rho)}}{\frac{ f_{n}(\rho)}{f_{n+1}(\rho)}-\frac{ f_{n-1}(\rho)}{f_{n}(\rho)}} \nonumber\\
   &= \sum_{j=1}^{\min \{n-1,i\}}a_j \frac{ \rho^j f_{n-j}(\rho) f_{n}(\rho) -\rho^j f_{n-1-j}(\rho)f_{n+1}(\rho)}{
f_{n}(\rho)^2- f_{n-1}(\rho)f_{n+1}(\rho)}+a_n \mathbf{1}_{\{n \geq i\}} \frac{ \rho^n f_{0}(\rho) f_{n}(\rho) }{
f_{n}(\rho)^2- f_{n-1}(\rho)f_{n+1}(\rho)}. \label{zhan}
\end{align}
Using the linearity of summation,  by (a) and (b), the result holds immediately.
\end{proof}

\begin{proof}[{\bf Proof of Theorem \ref{thm2}}]
By the sample path comparison, it is easy to have $n_s \leq n_e$, which ensures that $n_s$
is finite.
According to the definition of $n_s$, we have $S^r(n_s) \geq S^r(n_s-1)$ and $S^r(n_s)> S^r(n_s+1)$.
Using algebraic manipulations analogous to those in Section 2.4 in Hassin and Haviv \cite{Hassinbook},
it follows from (\ref{sr}) that these relations
can  also be rewritten as
$$\frac{\rho \sum_{i=1}^2 C_i \biggl[ \frac{\sum_{m=0}^{n_s} m^i \frac{\rho^m}{m!}}{\sum_{m=0}^{n_s+1} \frac{\rho^m}{m!}}-\frac{\sum_{m=0}^{n_s-1} m^i \frac{\rho^m}{m!}}{\sum_{m=0}^{n_s} \frac{\rho^m}{m!}}\biggl]}{\mathbb{E}(L_{n_s+1})-\mathbb{E}(L_{n_s})} > R \geq \frac{ \rho \sum_{i=1}^2 C_i \biggl[ \frac{\sum_{m=0}^{n_s-1} m^i \frac{\rho^m}{m!}}{\sum_{m=0}^{n_s} \frac{\rho^m}{m!}}-\frac{\sum_{m=0}^{n_s-2} m^i \frac{\rho^m}{m!}}{\sum_{m=0}^{n_s-1} \frac{\rho^m}{m!}}\biggl]}{\mathbb{E}(L_{n_s})-\mathbb{E}(L_{n_s-1})}. $$
By the results of Proposition \ref{mlem3}(c),
$\frac{\rho \sum_{i=1}^2 C_i \biggl[ \frac{\sum_{m=0}^{n} m^i \frac{\rho^m}{m!}}{\sum_{m=0}^{n+1} \frac{\rho^m}{m!}}-\frac{\sum_{m=0}^{n-1} m^i \frac{\rho^m}{m!}}{\sum_{m=0}^{n} \frac{\rho^m}{m!}}\biggl]}{\mathbb{E}(L_{n+1}  )-\mathbb{E}(L_{n})}$ is strictly increasing in $n$,
so $n_s$ is unique. Since $\frac{\rho \sum_{i=1}^2 C_i \biggl[ \frac{\sum_{m=0}^{n} m^i \frac{\rho^m}{m!}}{\sum_{m=0}^{n+1} \frac{\rho^m}{m!}}-\frac{\sum_{m=0}^{n-1} m^i \frac{\rho^m}{m!}}{\sum_{m=0}^{n} \frac{\rho^m}{m!}}\biggl]}{\mathbb{E}(L_{n+1} ) -\mathbb{E}(L_{n})}$ is strictly increasing in $\rho$, $n_s$ is decreasing in $\rho$.
\end{proof}

\begin{pro}\label{rtmm}
\begin{description}
\item[(a)]
$\mathbb{E}(L_{n})$ is strictly increasing in $n$.
  \item[(b)] $\mathbb{E}(L_{n+1})-\mathbb{E}(L_{n} )>\mathbb{E}(L_{n+2})-\mathbb{E}(L_{n+1})$.
     \item[(c)]
     $\mathbb{E}(L_{n})^2>\mathbb{E}(L_{n+1}  )\mathbb{E}(L_{n-1}),$ i.e.,
     $\frac{\mathbb{E}(L_{n} )}{\mathbb{E}(L_{n-1})}>\frac{\mathbb{E}(L_{n+1}  )}{\mathbb{E}(L_{n})}$.
\item[(d)]
$\frac{\mathbb{E}(L_{n+1})}{\mathbb{E}(L_{n})} $ is increasing in  $\rho$.
\end{description}
\end{pro}

\begin{proof}
(a) We use a coupling method here although we may also directly prove it by taking differences. Noting that $M/G/n/n$ and $M/M/n/n$ have the same expression of $\mathbb{E} (L_{n})$,  we only need to consider the Markovian case. Suppose that Process 1 and Process 2 is an $M/M/n/n$ queueing process
and an $M/M/n+1/n+1$ queueing process, respectively.
Follow the sample paths of two processes defined on
the same probability space and starting in the same state $s<n$, then both processes see the same arrivals, services for each customer, when customers in the system is no more than $n$.
Consider the first time the processes enter the state $n$ and a new customer arrives. Process 1 rejects
the customer but Process 2 accepts this customer.
At this point,  the queue length of the Process 2  is greater than the length queue of the Process 1.
Then, if Process 1 accepts an arriving customer, Process 2 must accept this customer.
If a service is the next event for Process 1, Process 2 also completes a service with
probability 1.  If only Process 1 completes a service for the ``$n+1$'' customer, both processes remain
coupled until the next time in state $n$ with an arriving customer.
Thus, the queue length of Process 1 is always less than the queue length of Process 2. Since the state
$n$ is positive recurrent for Process 2, the unequal relationship of  expected queue length must be strict, i.e., $ \mathbb{E}(L_{n})< \mathbb{E}(L_{n+1})$, which completes the proof.

(b)
We still use the coupling method and just consider three Markovian queueing processes.
Assume that
Process 1 (P1), Process 2 (P2) and Process 3 (P3) are an $M/M/n/n$, $M/M/n+1/n+1$ and $M/M/n+2/n+2$ queueing process, respectively.  Let $L_{Pi} (i=1,2,3)$ denote the queue length of Process $i$.
To complete the proof, follow the sample paths of three processes defined on
the same probability space and starting in the same state $s<n$. All processes move in parallel when the state is not in $n$. Consider the first time the processes enter the state $n$. If a service is the next event, all processes complete a service and  three processes remain coupled until the next time in state $n$. If an arrival is the next event, P1 rejects
the customer but P2 and P3 accept this customer. We denote this case by C1. Meanwhile, $L_{P2}-L_{P1}=1>0=L_{P3}-L_{P2}$.

Under C1, if a service for only P2 and P3 is the next event (the service rate of P2 and P3 is larger than the service rate of P1), three processes remain coupled and $L_{P2}-L_{P1}=0=L_{P3}-L_{P2}$.  If a service for P1, P2 and P3 is the next event, $L_{P2}-L_{P1}=1>0=L_{P3}-L_{P2}$.
If an arrival is the next event, P1, P2 rejects
the customer but P3 accepts this customer, thus $L_{P2}-L_{P1}=1=L_{P3}-L_{P2}$ and denote this case by C2.

Under C2, if an arrival is the next event, P1, P2 and P3 reject
the customer, thus $L_{P2}-L_{P1}=1=L_{P3}-L_{P2}$.
If a service for only P1 is the next event (the service rate of P2 and P3 is larger than the service rate of P1), we also have $L_{P2}-L_{P1}=1=L_{P3}-L_{P2}$.
If a service for only P2, P3 is the next event, we have $L_{P2}-L_{P1}=0<1=L_{P3}-L_{P2}$.
If a service for only P3 is the next event, we have $L_{P2}-L_{P1}=1>0=L_{P3}-L_{P2}$.
According to the Markovian property, the last two cases of C2 occur with the same probability and by the sample path comparison (analogous to the above comparison), before
both two cases return the same state, they have the same time path distribution. Thus, the expected difference of queue length are same.

Finally, all the above sample paths occur with positive probability, which immediately shows that $\mathbb{E}(L_{n+1})-\mathbb{E} (L_{n} )>\mathbb{E}(L_{n+2})-\mathbb{E}(L_{n+1})$.

(c)
It follows from (b) that
$$[2\mathbb{E}(L_{n})]^2 > [\mathbb{E}(L_{n-1})+\mathbb{E}(L_{n+1}  )]^2\geq[\mathbb{E}(L_{n-1})-\mathbb{E}(L_{n+1}  )]^2+4\mathbb{E}(L_{n-1})\mathbb{E}(L_{n+1}).$$
Thus, $\mathbb{E}(L_{n})^2>\mathbb{E}(L_{n+1}) \mathbb{E}(L_{n-1})$.

(d)
Using the expression of $\mathbb{E}(L_{n})$,
we only need to show that $\frac{ \bigl(\sum_{k=0}^{n} \frac{\rho^k}{k!}\bigl)^2}{\bigl(\sum_{k=0}^{n-1}\frac{\rho^k}{k!}\bigl)\bigl(\sum_{k=0}^{n+1}\frac{\rho^k}{k!}\bigl)}$ is increasing in $\rho$.
By rearranging terms,
we can expand the above expression with respect to $\rho$ as  follows:
$$\frac{ \bigl(\sum_{k=0}^{n}
\frac{\rho^k}{k!}\bigl)^2}{\bigl(\sum_{k=0}^{n-1}\frac{\rho^k}{k!}\bigl)\bigl(\sum_{k=0}^{n+1}\frac{\rho^k}{k!}\bigl)}=
\frac{\sum_{m=0}^{2n}\sum_{\substack{i+j=m\\ 0\leq i\leq n \\ 0 \leq j\leq n}}\frac{1}{i! j!}\rho^m}{\sum_{m=0}^{2n}\sum_{\substack{i+j=m\\ 0\leq i\leq n-1 \\ 0 \leq j\leq n+1}}\frac{1}{i! j!}\rho^m}
=\frac{\sum_{m=0}^{2n}\sum_{i= \max\{m-n,0\}}^{\min\{m,n\}}\frac{1}{i! (m-i)!}\rho^m}{\sum_{m=0}^{2n}\sum_{i= \max\{m-n-1,0\}}^{\min\{m,n-1\}}\frac{1}{i! (m-i)!}\rho^m}.$$
Note that for $m \leq n-1$, $$\frac{\sum_{i= \max\{m-n,0\}}^{\min\{m,n\}}\frac{1}{i! (m-i)!}}{\sum_{i= \max\{m-n-1,0\}}^{\min\{m,n-1\}}\frac{1}{i! (m-i)!}}=\frac{\sum_{i= 0}^{m}\frac{1}{i! (m-i)!}}{\sum_{i= 0}^{m}\frac{1}{i! (m-i)!}}=1.$$
For $m=n$, we have
$$\frac{\sum_{i= \max\{m-n,0\}}^{\min\{m,n\}}\frac{1}{i! (m-i)!}}{\sum_{i= \max\{m-n-1,0\}}^{\min\{m,n-1\}}\frac{1}{i! (m-i)!}}=
\frac{\frac{1}{n!(m-n)!}+\sum_{i=0}^{n-1}\frac{1}{i! (m-i)!}}{\sum_{i=0}^{n-1}\frac{1}{i! (m-i)!}}>1.$$
For $ n+1\leq m \leq2n$, it follows from $m-n<n+1$ that
$$h_m \triangleq\frac{\sum_{i= \max\{m-n,0\}}^{\min\{m,n\}}\frac{1}{i! (m-i)!}}{\sum_{i= \max\{m-n-1,0\}}^{\min\{m,n-1\}}\frac{1}{i! (m-i)!}}=
\frac{\frac{1}{n!(m-n)!}+\sum_{i= m-n}^{n-1}\frac{1}{i! (m-i)!}}{\sum_{i= m-n}^{n-1}\frac{1}{i! (m-i)!}+\frac{1}{(m-n-1)!(n+1)!}}.$$
Let $a_m=\sum_{i=m-n-1}^{n-1}\frac{1}{i!(m-i)!}$, then for $m\geq n+1$, we have
\begin{align}
  h_m &= \frac{a_m+\frac{1}{n!(m-n)!}-\frac{1}{(m-n-1)!(n+1)!}}{a_m} \nonumber \\
   &= 1+\frac{\frac{1}{n!(m-n)!}-\frac{1}{(m-n-1)!(n+1)!}}{a_m} \nonumber \\
   &= 1+\frac{2n-m+1}{(m-n)!(n+1)!a_m}.  \label{tebie}
\end{align}
Thus,
$$h_m-h_{m-1}= \frac{2n-m+1}{(m-n)!(n+1)!a_m} -\frac{2n-m+2}{(m-1-n)!(n+1)!a_{m-1}} \geq 0  $$
if and only if
$$a_{m-1} \geq a_{m}(m-n+\frac{1}{2n-m+1}).$$
Next, we prove that this relationship always holds for $n+1 \leq m \leq 2n$.

If $m=2n$, $a_m(m-n+\frac{1}{2n-m+1} ) =\frac{1}{(n-1)!n!}<\frac{1}{(n-2)!(n+1)!}+\frac{1}{(n-1)!n!}=a_{2n-1}$.

If $m=2n-1$,
\begin{align*}
  a_m(m-n+\frac{1}{2n-m+1} )  &= (\frac{1}{(n-1)!n!}+\frac{1}{(n-2)!(n+1)!})(n-1+\frac{1}{2}) \\
   &= \frac{1}{(n-2)!n!}+\frac{n-2+1}{(n-2)!(n+1)!}+\frac{1}{2}(\frac{1}{(n-1)!n!}+\frac{1}{(n-2)!(n+1)!})\\ &= \frac{1}{(n-2)!n!}+\frac{1}{(n-3)!(n+1)!}+\frac{3}{2(n-2)!(n+1)!}+\frac{1}{2}\frac{1}{(n-1)!n!}\\
    & \leq  \frac{1}{(n-1)!(n-1)!}+\frac{1}{(n-2)!n!}+\frac{1}{(n-3)!(n+1)!}\\
    &= a_{2n},
\end{align*}
where the last inequality is established from
$\frac{3(n-1)}{2(n+1)n}+\frac{1}{2n}\leq \frac{1}{2}$.

For $n+1 \leq m \leq 2n-2 $, if $2n-m+1$ is an odd number,
we have $\frac{1}{i!(m-i)!}< \frac{1}{(\frac{m}{2})!(\frac{m}{2})!}$ for $m-n-1 \leq i \leq n-1$, which implies that
\begin{eqnarray}
   a_m\frac{1}{2n-m+1}< \frac{2n-m+1}{(\frac{m}{2})!(\frac{m}{2})!} \frac{1}{2n-m+1}=\frac{1}{(\frac{m}{2})!(\frac{m}{2})!}.
\end{eqnarray}
Note that $(\frac{1}{(n-1)!(m-n+1)!}+\frac{1}{(m-n-1)!(n+1)!})(m-n)<(\frac{1}{(n-1)!(m-n)!}+\frac{1}{(m-n-2)!(n+1)!})$  for  $m<2n-2$, $\frac{m-n}{i!(m-i)!}<\frac{1}{i!(m-i-1)!}$ for $ \frac{m}{2}<i<n-1$ and $\frac{m-n}{i!(m-i)!} \leq\frac{1}{(i-1)!(m-i)!}$ for $m-n-1<i< \frac{m}{2}$, thus we also have
\begin{eqnarray*}
  a_{m}(m-n) < a_{m-1}-\frac{1}{(\frac{m}{2})!(\frac{m}{2})!}.
\end{eqnarray*}
For $n+1 \leq m \leq 2n-2 $, if $2n-m+1$ is an even number, we have $\frac{1}{i!(m-i)!}< \frac{1}{(\frac{m+1}{2})!(\frac{m-1}{2})!}$ for $m-n-1 \leq i \leq n-1$, which yields that
\begin{eqnarray}
   a_m\frac{1}{2n-m+1}&<& \frac{2n-m+1}{(\frac{m+1}{2})!(\frac{m-1}{2})!} \frac{1}{2n-m+1}=\frac{1}{(\frac{m+1}{2})!(\frac{m-1}{2})!}.
\end{eqnarray}
Note that when $m<2n-2$, $$(\frac{1}{(n-1)!(m-n+1)!}+\frac{1}{(m-n-1)!(n+1)!})(m-n)<(\frac{1}{(n-1)!(m-n)!}+\frac{1}{(m-n-2)!(n+1)!}),$$  $\frac{m-n}{i!(m-i)!}<\frac{1}{i!(m-i-1)!}$ for $ \frac{m+1}{2}<i<n-1$ and $\frac{m-n}{i!(m-i)!} \leq\frac{1}{(i-1)!(m-i)!}$ for $m-n-1<i< \frac{m+1}{2}$, which shows that

\begin{eqnarray*}
  a_{m}(m-n) &<& a_{m-1}-\frac{1}{(\frac{m+1}{2})!(\frac{m-1}{2})!}.\\
\end{eqnarray*}
Summarizing the above results, we know that $h_{m+1}\geq h_{m}$ for $m \geq n+1$.

Note that $\frac{1}{i!(n+1-i)!}=\frac{1}{i!(n-i)!}\frac{1}{n+1-i}<\frac{1}{i!(n-i)!}\frac{n}{n+1}$ for $0\leq i \leq n-1$, which, together with (\ref{tebie}), yields that for $m=n$,
\begin{eqnarray*}
  h_{n+1}-\frac{\sum_{i= \max\{m-n,0\}}^{\min\{m,n\}}\frac{1}{i! (m-i)!}}{\sum_{i= \max\{m-n-1,0\}}^{\min\{m,n-1\}}\frac{1}{i! (m-i)!}} &=& \frac{1}{n!}\frac{n}{(n+1)\sum_{i=0}^{n-1}\frac{1}{i!(n+1-i)!}} - \frac{1}{n!} \frac{1}{\sum_{i=0}^{n-1}\frac{1}{i!(n-i)!}}>0.
\end{eqnarray*}
Thus, $$\frac{\sum_{i= \max\{m-n,0\}}^{\min\{m,n\}}\frac{1}{i! (m-i)!}}{\sum_{i= \max\{m-n-1,0\}}^{\min\{m,n-1\}}\frac{1}{i! (m-i)!}}$$
is increasing in $m$ for $0 \leq m\leq 2n-2$. This, combining with Lemma \ref{lem4} implies that $\frac{\mathbb{E}(L_{n+1})}{\mathbb{E}(L_{n})} $ is increasing in  $\rho$. The proof is complete.
\end{proof}

\begin{proof}[{\bf Proof of Theorem \ref{thm3}}]
By the definition of $n_m$, we have $S(n_m) \geq S(n_m-1)$ and $S(n_m)> S(n_m+1)$.
Applying a similar argument to that in the proof of Theorem  \ref{thm2}, these relations
can  also be rewritten as
$$\frac{ \sum_{i=1}^2 C_i [
 \mathbb{E}(L_{n_m+1})(n_m)^{i}\!-\! \mathbb{E}(L_{n_m})(n_m \!-\!1)^{i}]}{ \mathbb{E}(L_{n_m+1})- \mathbb{E}(L_{n_m})}>\! R  \geq \frac{ \sum_{i=1}^2 C_i [ \mathbb{E}(L_{n_m})(n_m\!-\!1)^{i}\!-\! \mathbb{E}(L_{n_m-1})(n_m\!-\!2)^ {i}]}{ \mathbb{E}(L_{n_m})- \mathbb{E}(L_{n_m-1})}. $$
Note that
\begin{eqnarray*}
  \frac{   \mathbb{E}(L_{n+1})n^{i}- \mathbb{E}(L_{n})(n-1)^{i}}{ \mathbb{E}(L_{n+1})- \mathbb{E}(L_{n})} &=& n^i+ \frac{\mathbb{E}(L_{n} )(n^{i}-(n-1)^{i})}{ \mathbb{E}(L_{n+1})- \mathbb{E}(L_{n})}\\
   &=& (n+1)^i+ \frac{n^{i}-(n-1)^{i}}{
  \frac{\mathbb{E}(L_{n+1})}{\mathbb{E}(L_{n})}  - 1}.
\end{eqnarray*}
By Lemma \ref{rtmm}(c), $\frac{\mathbb{E}(L_{n+1}  )}{\mathbb{E}(L_{n})}$ is strictly decreasing in $n$ and $n^{i}-(n-1)^{i}$ is increasing in $n$ for $n \geq 1$, which means  that
$\frac{\mathbb{E}(L_{n+1})n^{i}- \mathbb{E}(L_{n} )(n-1)^{i}}{ \mathbb{E}(L_{n+1})- \mathbb{E}(L_{n})}$ is strictly increasing in $n$. Thus,
$n_m$ is unique. By Lemma \ref{rtmm}(d), we have
that
$ \frac{n^{i}-(n-1)^{i}}{
  \frac{\mathbb{E}(L_{n+1})}{\mathbb{E}(L_{n})}  - 1}$
is decreasing in $\rho$, which immediately shows that $n_m$ is increasing in $\rho$. The proof is complete.
\end{proof}

\begin{rmk}
Although the cost function in Theorem \ref{thm2} and Theorem \ref{thm3} is a combination of a linear function and a quadratic function, the methods in these two proofs are very general and suitable for any $i\geq 1$. Thus, the results of Theorem \ref{thm2} and Theorem \ref{thm3} can be  generalized to the case in which the cost is any finite polynomial function with non-negative coefficients.

\end{rmk}

\section{Conclusions and Extensions}

In this paper, we consider the equilibrium, social welfare, and revenue of an infinite-server queue in both observable and unobservable contexts and get the existence, uniqueness and computable expressions of optimal strategies for these goals.
We also numerically compare the social welfare and the revenue with different thresholds and information levels,
and insight into some useful information under different conditions.

On this topic, there is no denying that our hypothesis is somewhat rough compared to the actual background.
Because of this, many expansion  questions are worth studying. We make some comments on potential problems in the following.

\begin{itemize}
  \item[1.]
In the actual environment, the arrival of customers is affected by many aspects, such as weather or  holidays in the park examples.
Therefore,  analogous to Chen and Hasenbein \cite{Hasenbein}, it is interesting and practical to study the model with uncertain arrival rates.
 For unobservable case, we could investigate it by similar methods. However, for the observable case, affected by expectations, we still encounter some monotonic proofs that need to be solved urgently, although a large number of numerical results show that they are correct. We look forward to proving it in the future.

  \item[2.] In the notices posted in the system, we often see that  customers are non-homogeneous and the system has price discrimination.
      For example,  ticket prices of (toll)
parks are related to age groups, regions or other requirements.
       Therefore,
      it is a meaningful direction to research and design (pricing) the infinite-server queue with multiple types of customers. There is a lot of literature focusing on such queueing problems, such as Feinberg and Yang \cite{Feinberg}, Zhou, Chao and Gong \cite{Zhou},
      Liu and Hasenbein \cite{Liucp}, etc,
       so we believe that a similar method can be used to solve the infinite-server queue. Furthermore, considering the model of customers arriving in batches is also a more practical problem.

\end{itemize}

\section*{Acknowledgment}
This work is partially supported by the National Natural Science Foundation of China (No.
11771452, No. 11971486) and Natural Science Foundation of Hunan (No. 2020JJ4674).





\end{document}